\documentclass{amsart}

\usepackage{tikz}
\usepackage{xcolor}

\usepackage{latexsym,amsmath,amssymb,epic}
\usepackage{latexsym,amssymb,epic}
\usepackage{amsmath,amsthm}
\usepackage{color}

\usepackage{soul}

\usepackage{hyperref}
\hypersetup{
	colorlinks=true, 
	linktoc=all, 
	linkcolor=blue} 

\newtheorem{theorem}{Theorem}[section]
\newtheorem{lemma}[theorem]{Lemma}

\begin{document}

	\title[Congruences for $\overline{spt2}(n)$]{Extending Congruences for the number of smallest parts in overpartitions with smallest part even}

	\author{Robson da Silva}
	\address{Universidade Federal de S\~ao Paulo, S\~ao Jos\'e dos Campos, SP 12247--014, Brazil}
	\email{silva.robson@unifesp.br}

	

\maketitle
	
\begin{abstract}
In a recent paper, Jin, Liu, and Xia \cite{JLX} presented some modulo 4 congruences for $\overline{spt2}(n)$, the number of smallest parts in the overpartitions of $n$ where the smallest part is even and is not overlined. In this paper, we extend the list of such congruences in two directions. First, we prove some new individual congruences for $\overline{spt2}(n)$. Then, we provide a number of infinite families of Ramanujan-like congruences satisfied by $\overline{spt2}(n)$.
\end{abstract}

\section{Introduction}  

We recall that an overpartition \cite{CL} is a partition in which the first occurrence of each part may be overlined. Garvan and Jennings-Shaffer \cite{Garvan} introduced the function $\overline{spt2}(n)$, which is the number of smallest parts in the overpartitions of $n$ where the smallest part is even and is not overlined. For instance, considering the 14 overpartitions of $4$, namely
\begin{align*}
	&4, \, \overline{4}, \, 3+1, \, \overline{3}+1, \, 3 + \overline{1}, \, \overline{3}+\overline{1}, \, 2+2, \, \overline{2}+2, \, 2+1+1, \\
	& \overline{2}+1+1, \, 2 + \overline{1}+1, \, \overline{2}+\overline{1}+1, \, 1+1+1+1, \, \overline{1}+1+1+1,
\end{align*}
we see that $\overline{spt2}(4)=3$.

The study of arithmetic properties of partition functions emerged as a vibrant area of research as Ramanujan \cite{Ramanujan} proved several divisibility properties of $p(n)$, including:
$$\begin{array}{c}
	p(5n+4) \equiv 0 \ (\bmod \ 5), \\
	p(7n+5) \equiv 0 \  (\bmod \ 7), \\
	p(11n+6) \equiv 0 \  (\bmod \ 11), \\
\end{array}$$
where $p(n)$ is the number of partitions of $n$. After Ramanujan, significant achievements have been made by many authors for several functions that enumerate partitions satisfying certain restrictions. For example, Bringmann, Lovejoy, and Osburn \cite{BLO} proved that the partition function $\overline{spt2}(n)$ satisfies the following set of Ramanujan-like congruences: for all $n \geq 0$,
\begin{align*}
	\overline{spt2}(3n) \equiv 0 \pmod{3},\\
	\overline{spt2}(3n+1) \equiv 0 \pmod{3},\\
	\overline{spt2}(5n+3) \equiv 0 \pmod{5}.
\end{align*}

Very recently, Jin, Liu, and Xia \cite{JLX} presented some modulo 4 congruences for $\overline{spt2}(n)$, including
\begin{align}
	\overline{spt2}(8n+3) \equiv 0 \pmod{4}, \nonumber  \\
	\overline{spt2}(16n+14) \equiv 0 \pmod{4}, \label{J2} \\
	\overline{spt2}(32n+28) \equiv 0 \pmod{4}. \nonumber 
\end{align}

Our goal in this paper is to extend the congruences modulo $4$ satisfied by $\overline{spt2}(n)$. This is done in two directions. First, we prove some new individual congruences for $\overline{spt2}(n)$. Then, we provide a number of infinite families of Ramanujan-like congruences satisfied by $\overline{spt2}(n)$. The new individual congruences are the following.

\begin{theorem}
	For all $n \geq 0$,
	\begin{align}
		\overline{spt2}(36n+30) \equiv 0 \pmod{4}, \label{cg1} \\
		\overline{spt2}(48n+34) \equiv 0 \pmod{4}, \label{cg2} \\
		\overline{spt2}(72n+42) \equiv 0 \pmod{4}. \label{cg4} 
	\end{align}
	\label{Th1}
\end{theorem}

It turns out that much more is true about $\overline{spt2}(n)$ modulo 4, and our elementary approach to proving the congruences in Theorem \ref{Th1} yields the insights needed to see this. Thus, we state the following infinite families of non--nested Ramanujan--like congruences modulo 4 satisfied by $\overline{spt2}(n)$.

\begin{theorem}
	For all $j \geq 0$ and all $n \geq 0$,
	\begin{align}
		\overline{spt2}(2^{j+4}n+7\cdot 2^{j+1}) & \equiv 0 \pmod{4}, \label{if1} \\ 
		\overline{spt2}(2^{j+2}9n+15\cdot 2^{j+1}) & \equiv 0 \pmod{4}, \label{if2}\\
		\overline{spt2}(2^{j+4}3n+17\cdot 2^{j+1}) & \equiv 0 \pmod{4},  \label{if3} \\
		\overline{spt2}(2^{j+3}9n+21\cdot 2^{j+1}) &  \equiv 0 \pmod{4}, \label{if6} \\
		\overline{spt2}(2^{j+4}5n+17\cdot 2^{j+1}) & \equiv 0 \pmod{4},  \label{if4} \\
		\overline{spt2}(2^{j+4}5n+33\cdot 2^{j+1}) & \equiv 0 \pmod{4}.  \label{if5}
	\end{align}
	\label{Th2}
\end{theorem}

In the theorem below we present a different type of infinite family of Ramanujan-like congruence for $\overline{spt2}(n)$.

\begin{theorem} Let $p > 3$ be a prime such that $p \equiv 5 \mbox{ or } 7 \pmod{8}$. Then for all $k,m \geq 0$ with $p \nmid m$, 
	\begin{equation*}
		\overline{spt2}\left(32p^{2k+1}m + 24p^{2k+2} \right) \equiv 0 \pmod{4}.
	\end{equation*}
	\label{Th3}
\end{theorem}

This paper is organized as follows. In Section \ref{prelim}, we recall some basic identities that are used in the proof of Theorem \ref{Th1}. Section \ref{Thm1} is devoted to the proof of Theorem \ref{Th1}. The proofs of the infinite families of Ramanujan--like congruences shown in Theorem \ref{Th2} are presented in Section \ref{Thm2}. The proof of Theorem \ref{Th3} is the content of Section \ref{Thm3}.

\section{Preliminaries}
\label{prelim}

In order to accomplish the above goals, we require a few classical tools.  First, we recall Ramanujan's theta functions
\begin{align}
	f(a,b) & :=\sum_{n=-\infty}^\infty a^\frac{n(n+1)}{2}b^\frac{n(n-1)}{2}, \mbox{ for } |ab|<1, \nonumber \\
	\phi(q) & := f(q,q) = \sum_{n=-\infty}^{\infty} q^{n^2} = \frac{(q^2;q^2)_{\infty}^5}{(q;q)_{\infty}^2(q^4;q^4)_{\infty}^2}, \label{Rama1} 
	\\
	\psi(q) & := f(q,q^3) = \sum_{n=0}^{\infty} q^{n(n+1)/2} = \frac{(q^2;q^2)_{\infty}^2}{(q;q)_{\infty}}, 
	\label{Rama2}
\end{align}
where $(a;q)_\infty := (1-a)(1-aq)(1-aq^{2})\cdots$
is the usual {$q$-Pochhammer symbol. These functions satisfy many interesting properties (see Entries 18, 19, and 22 in \cite{Berndt}), including:
		\begin{align}
			\phi(-q) & = \frac{(q;q)_{\infty}^2}{(q^2;q^2)_{\infty}}, 
			\nonumber 
			\\
			\psi(-q) & = \frac{(q;q)_{\infty}(q^4;q^4)_{\infty}}{(q^2;q^2)_{\infty}}. \label{17.02.1}
		\end{align}
		
		Throughout the remainder of this paper, we define 
		$$f_k := (q^k;q^k)_{\infty}$$
		in order to shorten the notation.
		
		The proof of Theorem \ref{Th1} requires a few well-known 2- and 3-dissections.
		
		\begin{lemma} We have
			\begin{align}
				{f_{1}^{2}} & =  \frac{f_{2}f_8^5}{f_{4}^{2}f_{16}^{2}} -2q\frac{f_{2}f_{16}^{2}}{f_{8}}, \label{2diss1}  \\
				\frac{1}{f_{1}^{2}} & =  \frac{f_{8}^{5}}{f_{2}^{5}f_{16}^{2}} + 2q\frac{f_{4}^{2}f_{16}^{2}}{f_{2}^{5}f_{8}}. \label{2diss2}
			\end{align}
			\label{L1}
		\end{lemma}
		
		\begin{proof}
			By Entry 25 (i), (ii), (v), and (vi) in \cite[p. 40]{Berndt}, we have
			\begin{equation}
				\phi(q) = \phi(q^4) + 2q\psi(q^8),
				\label{4.12}
			\end{equation}
			and
			\begin{equation}
				\phi(q)^2 = \phi(q^2)^2 + 4q\psi(q^4)^2.
				\label{Eq1}
			\end{equation}
			Identity in \eqref{2diss1} follows from \eqref{Eq1} after replacing $q$ by $-q$.
			By \eqref{Rama1} and \eqref{Rama2} we can rewrite \eqref{4.12} in the form
			\begin{equation*}
				\frac{f_{2}^{5}}{f_{1}^{2}f_{4}^{2}} = \frac{f_{8}^{5}}{f_{4}^{2}f_{16}^{2}} + 2q\frac{f_{16}^{2}}{f_{8}},
			\end{equation*}
			from which we obtain \eqref{2diss2} after multiplying both sides by $\frac{f_4^2}{f_2^5}$.
		\end{proof}
		
		\begin{lemma} We have
			\begin{align}
				f_{1}^{3} & = \displaystyle\frac{f_{6}f_{9}^6}{f_{3}f_{18}^3} + 4q^3\displaystyle\frac{f_{3}^2f_{18}^{6}}{f_{6}^2f_{9}^3} - 3qf_{9}^{3}, \label{3diss1} \\
				\displaystyle\frac{f_{4}}{f_{1}} & = \displaystyle\frac{f_{12}f_{18}^{4}}{f_{3}^3f_{36}^2} + q\displaystyle\frac{f_{6}^2f_{9}^{3}f_{36}}{f_{3}^4f_{18}^2} + 2q^2 \displaystyle\frac{f_{6}f_{18}f_{36}}{f_{3}^3}, \label{3diss2} \\
				f_{1}f_{2} & =  \displaystyle\frac{f_{6}f_{9}^{4}}{f_{3}f_{18}^2} - qf_{9}f_{18} -2q^2 \displaystyle\frac{f_{3}f_{18}^{4}}{f_{6}f_{9}^2}. \label{3diss3}
			\end{align}
			\label{L2}
		\end{lemma}
		
		\begin{proof} Identity \eqref{3diss1} appears in \cite[Eq. (14.8.5)]{H}. A proof of \eqref{3diss2} can be found in \cite{Yao}. 
			The identity \eqref{3diss3} was proven in \cite{Sellers2}.
		\end{proof}
		
		\begin{lemma} We have
			\begin{equation}
				\psi(q) = \displaystyle\frac{f_{2}^2}{f_{1}} =  \displaystyle\frac{f_{6}f_{9}^{2}}{f_{3}f_{18}} + q\displaystyle\frac{f_{18}^{2}}{f_{9}}. \label{3dissPsi}
			\end{equation}
			\label{L3}
		\end{lemma}
		
		\begin{proof}
			A proof of this identity can be seen in \cite[Eq. (14.3.3)]{H}.
		\end{proof}
		
		\begin{lemma} We have
			\begin{equation*}
				\psi(-q) = \displaystyle\frac{f_{3}f_{12}f_{18}^5}{f_{6}^2f_9^2f_{36}^2} - q\displaystyle\frac{f_{9}f_{36}}{f_{18}}.    
			\end{equation*}
			\label{L4}
		\end{lemma}
		
		\begin{proof}
			A proof of this identity can be seen in \cite[Lemma 2]{RS}.
		\end{proof}
		
		\begin{lemma} We have
			\begin{equation*}
				\phi(-q) =  \displaystyle\frac{f_{9}^2}{f_{18}} - 2q\displaystyle\frac{f_{3}f_{18}^2}{f_{6}f_{9}}
			\end{equation*}
			\label{L5}
		\end{lemma}
		
		\begin{proof}
			A proof of this identity can be seen in \cite[Eq. (14.3.4)]{H}.
		\end{proof}
		
		To close this section, we recall the following identity, which appears in \cite[Eq. (1.7.1)]{H}.
		
		\begin{lemma} We have
			\begin{equation*}
				f_1^3 = \displaystyle\sum_{n=0}^{\infty} (-1)^n (2n+1) q^{(n^2+n)/2}.
			\end{equation*}
			\label{L6}
		\end{lemma}


		\section{Proof of Theorem \ref{Th1}}
		\label{Thm1}
		
		In many occasions below, we will make use of the following elementary fact: for any positive integers $k, m$, and $n$,
		$$f_{n}^{2^km} \equiv f_{2n}^{2^{k-1}m} \pmod{2^k}.$$
		
		We know from \cite[Eq. (4.12)]{JLX} that
		\begin{align}
			\sum_{n=0}^{\infty}\overline{spt2}(4n+2)q^n & \equiv 2\frac{f_8^2}{f_4}-\frac{f_2^3f_4^2}{f_1^2} \nonumber \\
			& \equiv 2f_4^3 - f_2^3 \left( \frac{f_4}{f_1} \right)^2 \pmod{4}. \label{eq1}
		\end{align}
		Using \eqref{3diss1} and \eqref{3diss2}, we obtain
		\begin{align*}
			\sum_{n=0}^{\infty}\overline{spt2}(4n+2)q^n & \equiv 2\frac{f_{24}f_{36}^6}{f_{12}f_{72}^3} + 2q^4f_{36}^3 - \left( \frac{f_{12}f_{18}^6}{f_{6}f_{36}^3} - 3q^2f_{18}^3 \right)\left( \frac{f_{12}^2f_{18}^8}{f_{3}^6f_{36}^4} \right. \\ 
			& \left. \ \ \  +q^2\frac{f_6^4f_{9}^6f_{36}^2}{f_{3}^8f_{18}^2} + 2q\frac{f_{6}^2f_{9}^3f_{12}f_{18}^2}{f_{3}^7f_{36}} \right) \pmod{4},
		\end{align*}
		which, after extracting the terms of the form $q^{3n+1}$, yields
		\begin{equation*}
			\sum_{n=0}^{\infty}\overline{spt2}(12n+6)q^{3n+1} \equiv 2q^4f_{36}^3 + 2q\frac{f_{6}f_{9}^3f_{12}^2f_{18}^8}{f_{3}^7f_{36}^4} +3q^4\frac{f_{6}^4f_{9}^6f_{18}^3f_{36}^2}{f_{3}^8f_{18}^2} \pmod{4}.
		\end{equation*}
		Dividing by $q$ and replacing $q^3$ by $q$, we obtain
		\begin{align*}
			\sum_{n=0}^{\infty}\overline{spt2}(12n+6)q^{n} & \equiv 2qf_{12}^3 + 2\frac{f_{2}f_{3}^3f_{4}^2f_{6}^8}{f_{1}^7f_{12}^4} +3q\frac{f_{2}^4f_{3}^6f_{6}^3f_{12}^2}{f_{1}^8f_{6}^2} \\
			& \equiv 2qf_{12}^3 + 2\frac{f_3^3f_4^2}{f_1f_2^2} + 3qf_3^2f_6^5 \\
			& \equiv  2qf_{12}^3 +3qf_3^2f_6^5 + 2f_3^3 \frac{f_2^2}{f_1} \pmod{4}.
		\end{align*}
		Thanks to Lemma \ref{L3}, we see that the right-hand side of the congruence above does not have terms of the form $q^{3n+2}$, which proves \eqref{cg1}.
		
		We now proceed to the proof of \eqref{cg2}. We start by using \eqref{2diss2} to extract the even part of \eqref{eq1}, which yields
		\begin{align*}
			\sum_{n=0}^{\infty}\overline{spt2}(8n+2)q^{2n} & \equiv 2f_4^3 - \frac{f_4^2f_8^5}{f_2^2f_{16}^2} \\
			& \equiv 2f_4^3 -f_2^2f_8 \pmod{4}.
		\end{align*}
		After replacing $q^2$ by $q$, we get
		\begin{align}
			\sum_{n=0}^{\infty}\overline{spt2}(8n+2)q^{n} & \equiv 2f_2^3 -f_1^2f_4 \pmod{4}. \label{eq2}
		\end{align}
		Now we make use of \eqref{2diss1} to extract the even parts on both sides of the above congruence:
		\begin{align*}
			\sum_{n=0}^{\infty}\overline{spt2}(16n+2)q^{2n} & \equiv 2f_2^3 -\frac{f_2f_8^5}{f_4f_{16}^2} \pmod{4}.
		\end{align*}
		After replacing $q^2$ by $q$, we are left with
		\begin{align}
			\sum_{n=0}^{\infty}\overline{spt2}(16n+2)q^{n} & \equiv 2f_1^3 -\frac{f_1f_4^5}{f_2f_{8}^2} \nonumber \\
			& \equiv 2f_1^3 - \frac{f_1f_4}{f_2} \nonumber \\ 
			& \equiv 2f_1^3 - \psi(-q) \pmod{4}, \label{eq3}
		\end{align}
		using \eqref{17.02.1}. Thanks to \eqref{3diss1} and Lemma \ref{L4} we see that the right-hand side of the last congruence does not have terms of the form $q^{3n+2}$, which completes the proof of \eqref{cg2}.

		In order to prove \eqref{cg4}, we rewrite \eqref{eq2} as
		\begin{align*}
			\sum_{n=0}^{\infty}\overline{spt2}(8n+2)q^{n} & \equiv 2f_2^3 -\frac{f_1^2}{f_2}f_2f_4 \pmod{4}.
		\end{align*}
		Combining \eqref{3diss1}, Lemma \ref{L5}, and \eqref{3diss3} we can extract the terms of the form $q^{3n+2}$ on both sides of the last congruence, which yields the following:
		\begin{align*}
			\sum_{n=0}^{\infty}\overline{spt2}(24n+18)q^{3n+2} & \equiv 2q^2f_{18}^3 + q^2f_9^2f_{36} \pmod{4}.
		\end{align*}
		Dividing by $q^2$ and replacing $q^3$ by $q$, we are left with
		\begin{align*}
			\sum_{n=0}^{\infty}\overline{spt2}(24n+18)q^{n} & \equiv 2f_{6}^3 + f_3^2f_{12} \pmod{4}.
		\end{align*}
		Since there are no terms of the form $q^{3n+1}$ in $2f_{6}^3 + f_3^2f_{12}$, we obtain \eqref{cg4}.

		
		\section{Proof of Theorem \ref{Th2}}
		\label{Thm2}
		
		We begin by recalling an interesting internal congruence satisfied by $\overline{spt2}(n)$ (see \cite[Theorem 1.7]{JLX}) that will serve as an important component in our remaining proofs: for $n \geq 0$,
		\begin{equation}
			\overline{spt2}(8n+5) \equiv \overline{spt2}(16n+10) \pmod{4}.
			\label{IntCong1}
		\end{equation}

		\subsection{Proof of \eqref{if1}}
		
		The congruence in \eqref{if1} follows from a straightforward proof by induction on $j.$  
		First, we note that the basis case, $j=0,$ is simply the statement that, for all $n\geq 0,$ 
		$$
		\overline{spt2}(16n+14) \equiv 0 \pmod{4}.
		$$
		This is Equation \eqref{J2} above, which appears in \cite[Theorem 1.8]{JLX}.  
		
		Next, we assume that the statement is true for some $j\geq 0.$  We then wish to prove that 
		$$
		\overline{spt2}\left( 2^{j+5}n + 7\cdot 2^{j+2} \right) \equiv 0 \pmod{4}
		$$
		for all $n\geq 0$ as well.  Note that 
		\begin{align*}
			2^{j+5}n + 7\cdot 2^{j+2} & =
			16\left( 2^{j+1}n + 7\cdot 2^{j-2} - \frac{5}{8} \right) + 10.
		\end{align*}
		Thanks to (\ref{IntCong1}), we know that, for all $n\geq 0,$
		$$\overline{spt2}(8n+5) \equiv \overline{spt2}(16n+10) \pmod{4}$$ 
		Thus, 
		\begin{align*}
			\overline{spt2}\left( 2^{j+5}n + 7\cdot 2^{j+2} \right)
			&=
			\overline{spt2}\left( 16\left( 2^{j+1}n + 7\cdot 2^{j-2} - \frac{5}{8} \right) + 10 \right)
			\\
			& \equiv 
			\overline{spt2}\left( 8\left( 2^{j+1}n + 7\cdot 2^{j-2} - \frac{5}{8} \right) + 5 \right) \pmod{4}\\
			&= 
			\overline{spt2}\left( 2^{j+4}n + 7\cdot 2^{j+1} \right) \\
			&\equiv 0 \pmod{4} 
		\end{align*}
		by the induction hypothesis.  The result follows.

		\subsection{Proof of \eqref{if2}}
		
		The proof of \eqref{if1} also follows from a straightforward proof by induction on $j.$  
		First, we note that the basis case, $j=0,$ is the same as saying that, for all $n\geq 0,$ 
		$$
		\overline{spt2}(36n+30) \equiv 0 \pmod{4}.
		$$
		This is Equation \eqref{cg1} in Theorem \ref{Th1}.  
		
		Next, we assume that the statement is true for some $j\geq 0.$  We then wish to prove that 
		$$
		\overline{spt2}\left( 2^{j+3}9n + 15\cdot 2^{j+2} \right) \equiv 0 \pmod{4}
		$$
		for all $n\geq 0$ as well.  Note that 
		\begin{align*}
			2^{j+3}9n + 15\cdot 2^{j+2} & =
			16\left( 2^{j-1}9n + 15\cdot 2^{j-2} - \frac{5}{8} \right) + 10.
		\end{align*}
		Thanks to (\ref{IntCong1}), we have 
		\begin{align*}
			\overline{spt2}\left( 2^{j+3}9n + 15\cdot 2^{j+2} \right)
			&=
			\overline{spt2}\left( 16\left( 2^{j-1}9n + 15\cdot 2^{j-2} - \frac{5}{8} \right) + 10 \right)
			\\
			& \equiv 
			\overline{spt2}\left( 8\left( 2^{j-1}9n + 15\cdot 2^{j-2} - \frac{5}{8} \right) + 5 \right) \pmod{4}\\
			&= 
			\overline{spt2}\left( 2^{j+2}9n + 15\cdot 2^{j+1} \right) \\
			&\equiv 0 \pmod{4} 
		\end{align*}
		by the induction hypothesis, which completes the proof of \eqref{if2}.

		\subsection{Proof of \eqref{if3}}
		
		Here we also proceed by induction on $j.$ The basis case, $j=0,$ is the same as saying that, for all $n\geq 0,$ 
		$$
		\overline{spt2}(48n+34) \equiv 0 \pmod{4}.
		$$
		This is Equation \eqref{cg2} in Theorem \ref{Th1}.  
		
		Next, we assume that the statement is true for some $j\geq 0.$  Then we prove that 
		$$
		\overline{spt2}\left( 2^{j+5}3n + 17\cdot 2^{j+2} \right) \equiv 0 \pmod{4}
		$$
		for all $n\geq 0$ as well.  Note that 
		\begin{align*}
			2^{j+5}3n + 17\cdot 2^{j+2} & =
			16\left( 2^{j+1}3n + 17\cdot 2^{j-2} - \frac{5}{8} \right) + 10.
		\end{align*}
		Thanks to (\ref{IntCong1}), we have 
		\begin{align*}
			\overline{spt2}\left( 2^{j+5}3n + 17\cdot 2^{j+2} \right)
			&=
			\overline{spt2}\left( 16\left( 2^{j+1}3n + 17\cdot 2^{j-2} - \frac{5}{8} \right) + 10 \right)
			\\
			& \equiv 
			\overline{spt2}\left( 8\left( 2^{j+1}3n + 17\cdot 2^{j-2} - \frac{5}{8} \right) + 5 \right) \pmod{4}\\
			&= 
			\overline{spt2}\left( 2^{j+4}3n + 17\cdot 2^{j+1} \right) \\
			&\equiv 0 \pmod{4} 
		\end{align*}
		by the induction hypothesis, which completes the proof of \eqref{if3}.

		\subsection{Proof of \eqref{if6}}
		
		Proceeding by induction on $j$, we see that the basis case, $j=0,$ is the same as saying that, for all $n\geq 0,$ 
		$$
		\overline{spt2}(72n+42) \equiv 0 \pmod{4}.
		$$
		This is Equation \eqref{cg4} in Theorem \ref{Th1}.  
		
		Next, we assume that the statement is true for some $j\geq 0.$  Then we prove that 
		$$
		\overline{spt2}\left( 2^{j+4}9n + 21\cdot 2^{j+2} \right) \equiv 0 \pmod{4}
		$$
		for all $n\geq 0$ as well.  Note that 
		\begin{align*}
			2^{j+4}9n + 21\cdot 2^{j+2} & =
			16\left( 2^{j}9n + 21\cdot 2^{j-2} - \frac{5}{8} \right) + 10.
		\end{align*}
		Thanks to (\ref{IntCong1}), we have 
		\begin{align*}
			\overline{spt2}\left( 2^{j+4}9n + 21\cdot 2^{j+2} \right)
			&=
			\overline{spt2}\left( 16\left( 2^{j}9n + 21\cdot 2^{j-2} - \frac{5}{8} \right) + 10 \right)
			\\
			& \equiv 
			\overline{spt2}\left( 8\left( 2^{j}9n + 21\cdot 2^{j-2} - \frac{5}{8} \right) + 5 \right) \pmod{4}\\
			&= 
			\overline{spt2}\left( 2^{j+3}9n + 21\cdot 2^{j+1} \right) \\
			&\equiv 0 \pmod{4} 
		\end{align*}
		by the induction hypothesis, which completes the proof of \eqref{if6}.
		
		\subsection{Proof of \eqref{if4}}
		
		Proceeding by induction on $j$, we see that the basis case, $j=0,$ is the same as saying that, for all $n\geq 0,$ 
		$$
		\overline{spt2}(80n+34) \equiv 0 \pmod{4}.
		$$
		This follows directly from \cite[Theorem 1.4]{JLX}. Indeed, since $k(k+1)/2 \not\equiv 2 \pmod{5}$, we see that 
		$$\overline{spt2}(80n+34) = \overline{spt2}(16\cdot (5n+2)+2) \equiv 0 \pmod{4}$$ 
		by \cite[Eq. (1.4)]{JLX}.}

	Next, we assume that the statement is true for some $j\geq 0.$  Then we prove that 
	$$
	\overline{spt2}\left( 2^{j+5}5n + 17\cdot 2^{j+2} \right) \equiv 0 \pmod{4}
	$$
	for all $n\geq 0$ as well.  Note that 
	\begin{align*}
		2^{j+5}5n + 17\cdot 2^{j+2} & =
		16\left( 2^{j+1}5n + 17\cdot 2^{j-2} - \frac{5}{8} \right) + 10.
	\end{align*}
	Thanks to (\ref{IntCong1}), we have 
	\begin{align*}
		\overline{spt2}\left( 2^{j+5}5n + 17\cdot 2^{j+2} \right)
		&=
		\overline{spt2}\left( 16\left( 2^{j+1}5n + 17\cdot 2^{j-2} - \frac{5}{8} \right) + 10 \right)
		\\
		& \equiv 
		\overline{spt2}\left( 8\left( 2^{j+1}5n + 17\cdot 2^{j-2} - \frac{5}{8} \right) + 5 \right) \pmod{4}\\
		&= 
		\overline{spt2}\left( 2^{j+4}5n + 17\cdot 2^{j+1} \right) \\
		&\equiv 0 \pmod{4} 
	\end{align*}
	by the induction hypothesis, which completes the proof of \eqref{if4}.
	
	\subsection{Proof of \eqref{if5}}
	
	Proceeding by induction on $j$, we see that the basis case, $j=0,$ is the same as saying that, for all $n\geq 0,$ 
	$$
	\overline{spt2}(80n+66) \equiv 0 \pmod{4}.
	$$
	This follows directly from \cite[Theorem 1.4]{JLX}.  Indeed, since $k(k+1)/2 \not\equiv  4 \pmod{5}$, we see that 
	$$\overline{spt2}(80n+66) = \overline{spt2}(16\cdot (5n+4)+2) \equiv 0 \pmod{4},$$ 
	by \cite[Eq. (1.4)]{JLX}.
	
	Next, we assume that the statement is true for some $j\geq 0.$  Then we prove that 
	$$
	\overline{spt2}\left( 2^{j+5}5n + 33\cdot 2^{j+2} \right) \equiv 0 \pmod{4}
	$$
	for all $n\geq 0$ as well.  Note that 
	\begin{align*}
		2^{j+5}5n + 33\cdot 2^{j+2} & =
		16\left( 2^{j+1}5n + 33\cdot 2^{j-2} - \frac{5}{8} \right) + 10.
	\end{align*}
	Thanks to (\ref{IntCong1}), we have 
	\begin{align*}
		\overline{spt2}\left( 2^{j+5}5n + 33\cdot 2^{j+2} \right)
		&=
		\overline{spt2}\left( 16\left( 2^{j+1}5n + 33\cdot 2^{j-2} - \frac{5}{8} \right) + 10 \right)
		\\
		& \equiv 
		\overline{spt2}\left( 8\left( 2^{j+1}5n + 33\cdot 2^{j-2} - \frac{5}{8} \right) + 5 \right) \pmod{4}\\
		&= 
		\overline{spt2}\left( 2^{j+4}5n + 33\cdot 2^{j+1} \right) \\
		&\equiv 0 \pmod{4} 
	\end{align*}
	by the induction hypothesis, which completes the proof of \eqref{if4}.


	\section{Proof of Theorem \ref{Th3}}
	\label{Thm3}
	
	Initially we recall the following congruence (see \cite[Eq. (4.16)]{JLX}):
	\begin{align*}
		\sum_{n=0}^{\infty}\overline{spt2}(16n+8)q^{n} & \equiv -\frac{f_2^3f_4^2}{f_1^2} \pmod{4}.
	\end{align*}
	Using \eqref{2diss2}, we get
	\begin{align*}
		\sum_{n=0}^{\infty}\overline{spt2}(32n+24)q^{2n+1} & \equiv 2q\frac{f_4^4f_{16}^2}{f_2^2f_8} \pmod{4},
	\end{align*}
	which, after dividing by $q$ and replacing $q^2$ by $q$, yields
	\begin{align}
		\sum_{n=0}^{\infty}\overline{spt2}(32n+24)q^{n} & \equiv 2\frac{f_2^4f_{8}^2}{f_1^2f_4} \equiv 2f_2^3f_4^3 \pmod{4}.
		\label{eq4}
	\end{align}
	
	Using Lemma \ref{L6}, it follows from \eqref{eq4} that
	\begin{align*}
		\sum_{n=0}^{\infty} \overline{spt2}(32n+24)q^{n} & \equiv 2 \sum_{k=0}^{\infty}\sum_{l = 0}^{\infty} q^{k(k+1)+2l(l+1)} \pmod{4}.
	\end{align*}
	Note that $n = k(k+1)+2l(l+1)$ is equivalent to writing $4n+3 = (2k+1)^2 + 2(2l+1)^2$. Thus, $\overline{spt2}(32n+24) \equiv 0 \pmod{4}$ if $4n+3$ is not of the form $x^2+2y^2$. Since $p \equiv5 \mbox{ or } 7 \pmod{8}$ we know that $\displaystyle \left( \frac{-2}{p}\right) = -1$, i.e., $-2$ is a quadratic nonresidue modulo $p$. This implies that $\nu_p(N)$ is even if $N$ is of the form $2x^2+y^2$. However, here $n = p^{2k+1}m + 3(p^{2k+2}-1)/{4}$. Then
	$$4n+3= 4p^{2k+1}m + 3p^{2k+2} = p^{2k+1}(4m+3p).$$
	Therefore $\nu_p(4n+3)$ is odd and the result follows.

	\section{Concluding remarks	}\label{sec13}
	
	
	In the work above, we have accomplished our primary goal, which was to prove new Ramanujan--like congruences satisfied by $\overline{spt2}(n)$ modulo $4$, including some infinite families of congruences. With that said, we admit that our results above are not exhaustive; computational evidence indicates that there are more Ramanujan--like congruences modulo $4$  satisfied by $\overline{spt2}(n)$ which are not covered by our results above. We leave it to the interested reader to consider proving additional Ramanujan--like congruences which are satisfied by $\overline{spt2}(n)$.



\begin{thebibliography}{00}
\bibitem{Berndt} B. C. Berndt, {Ramanujan's Notebooks}, Part III. Springer, New York, 1991.
\bibitem{BLO} K. Bringmann, J. Lovejoy, R. Osburn, Rank and crank moments for overpartitions, {\it J. Number Theory} {\bf 129} (2009), 1758--1772.
\bibitem{CL} S. Corteel, J. Lovejoy, Overpartitions, {\it Trans. Amer. Math. Soc.} {\bf 356} (2004), 1623--1635.
\bibitem{RS} R. da Silva, J. A. Sellers, {Infinitely Many Congruences for k-Regular Partitions with Designated Summands}. {\it Bull. Braz. Math. Soc.} {\bf 51} (2020), 357--370.
\bibitem{Garvan} F. Garvan, C. Jennings-Shaffer, The spt-crank for overpartitions, {\it Acta Arith.} {\bf 166} (2014), 141--188.
\bibitem{H} M. D. Hirschhorn, {The power of $q$, a personal journey}, Developments in Mathematics, v. 49, Springer, 2017.
\bibitem{Sellers2} M. D. Hirschhorn, J. A. Sellers, {A Congruence Modulo 3 for Partitions into Distinct Non-Multiples of Four}. {\it J. Integer Sequences} {\bf 17} (2014), Article 14.9.6.
\bibitem{JLX} J. Jin, E. H. Liu, E. X. W. Xia, Characterizations of congruences modulo 2 and 4 for the number of smallest parts in overpartitions with smallest part even. {\it Ramanujan J.} {\bf 67} (2025), 99. 
\bibitem{Ramanujan} Ramanujan, S.: {Collected Papers}, Cambridge University Press, London, 1927; reprinted: A. M. S. Chelsea, 2000 with new preface and extensive commentary by B. Berndt.
\bibitem{Yao} O. X. M. Yao, {The relations between $N(a,b,c,d;n)$ and $t(a,b,c,d;n)$ and $(p,k)$-parametrization of theta functions}. {\it J. Math. Anal. Appl.} {\bf 453} (2017), 125--143.
\end{thebibliography}
\end{document}